\makeatletter

\documentclass[12pt]%
{article}

\newsavebox{\savepar}

\usepackage{amsmath}
\usepackage{amsfonts}
\usepackage{euscript}
\usepackage{oldgerm}
\usepackage{amsthm}
\usepackage{rotating}
\usepackage{mathrsfs}
\usepackage{hyperref}
\usepackage{amssymb}
\usepackage{epsfig}
\usepackage{dsfont}
\usepackage{subfig}
\makeindex

\newtheorem{theorem}{Theorem}[section]

\theoremstyle{definition}
\newtheorem{definition}[theorem]{Definition}

\theoremstyle{remark}

\newtheorem{proposition}[theorem]{\bf Proposition}

\theoremstyle{definition}

\theoremstyle{remark}

\begin{document}
	
	\newcommand{\norm}[1]{\left\lVert #1\right\rVert}
	\newcommand{\namelistlabel}[1]{\mbox{#1}\hfil}
	\newenvironment{namelist}[1]{%
		\begin{list}{}
			{
				\let\makelabel\namelistlabel
				\settowidth{\labelwidth}{#1}
				\setlength{\leftmargin}{1.1\labelwidth}
			}
		}{%
	\end{list}}

	\newcommand{\K}{\mathcal K}
	\newcommand{\inp}[2]{\langle {#1} ,\,{#2} \rangle}
	\newcommand{\vspan}[1]{{{\rm\,span}\{ #1 \}}}
	\newcommand{\R} {{\mathbb{R}}}
	
	\newcommand{\B} {{\mathcal{B}}}
	\newcommand{\C} {{\mathbb{C}}}
	\newcommand{\N} {{\mathbb{N}}}
	\newcommand{\Q} {{\mathbb{Q}}}
	\newcommand{\LL} {{\mathbb{L}}}
	\newcommand{\Z} {{\mathbb{Z}}}

	\title{Spectral Theorem for a bounded self adjoint operator on a Bicomplex Hilbert space}
	\author{Akshay S. RANE\footnote{Department of Mathematics, Institute of Chemical Technology, Nathalal Parekh Marg, Matunga, Mumbai 400 019, India, email :  as.rane@ictmumbai.edu.in}  ~ and ~ Mandar THATTE\footnote{Department of Mathematics, Institute of Chemical Technology, Nathalal Parekh Marg, Matunga, Mumbai 400 019, India, email :  007thattecharlie@gmail.com } 
		\hspace {1mm}
	}
	\date{ }
	\maketitle

	\begin{abstract}
		In this paper,  we shall consider the notion of bicomplex inner product and define bicomplex Hilbert space. We shall define $L^{2}[a,b]$ where the functions take bicomplex values. We shall prove the Theorem for a bounded self adjoint operator on a bicomplex Hilbert space which is not compact.
	\end{abstract}
	

	\newpage
	\newpage
	
	
	\setcounter{equation}{0}
	\section{Introduction}
	The  notion of Bicomplex numbers and hyperbolic has been introduced in \cite{ALSS}. They have also discussed the idempotent representation of bicomplex numbers. Fundamental theorems in functional analysis like the Baire Category theorem, open mapping theorem, closed graph theorem, uniform boundedness principle for modules over bicomplex/hyperbolic numbers have been proved in \cite{HAR}. Zabreiko's lemma and as a consequence the fundamental theorems have been proved in \cite{ASR}.The notion of bicomplex inner product and definition of adjoint operator has been dicussed in \cite{Ger} and \cite{Ger2}. The notions of spectrum and resolvent operator in case of a bicomplex Banach module is introduced in \cite{CSS}. The spectral decomposition theorem for a compact self adjoint operator on a bicomplex Hilbert space has been proved in \cite{Cha}. Spectral theorem for a self adjoint on the usual Hilbert space is well known. One such reference is\cite{BR}. In this paper, we shall first define the space $L^2[a,b]$, where the functions defined on $[a,b]$ takes bicomplex values. We define the inner product on this space consistent as what is done for $l_2$ in \cite{Cha2}. Later we shall prove the spectral theorem for a bounded self adjoint operator on a Bicomplex Hilbert space which is not compact.
	The set of bi complex numbers is defined as $$ \mathbb{B} \mathbb{C}= \lbrace Z= z_{1} + z_{2}j | z_{1},z_{2} \in C(i) \rbrace $$ where $i$, $j$ are such that $ij=ji$ and $i^{2}=j^{2}=-1.$ 
	The set $\mathbb{B} \mathbb{C}$ forms a Ring under the usual addition and multiplication. The product of imaginary units $i$ and $j$ defines a hyperbolic unit $k$ such that $k^{2}=1.$
	Since $ \mathbb{C}$ contains two imaginary units whose square is -1 and one hyperbolic unit whose square is 1, we consider three conjugations for bi complex numbers.
	\begin{enumerate}
		\item $\overline{Z}= \overline{z_{1}}+\overline{z_{2}}j$  (\textbf{bar-conjugation}).
		\item $Z^{+}= z_{1}-z_{2}j$  (\textbf{$+$ - conjugation}).
		\item $Z^{*}= \overline{z_{1}}-\overline{z_{2}}j$  (\textbf{*- conjugation}).
	\end{enumerate}
	The set $\mathbb{D}$ of hyperbolic numbers is defined as $$ \mathbb{D} = \lbrace h = h_{1} + k h_{2} | h_{1},h_{2} \in \mathbb{R} \rbrace.$$ The set $\mathbb{D}$ is a ring and module over itself. Also, $$ \mathbb{D^{+}} =\lbrace h =e_{1}a_{1}+ e_{2}a_{2} | a_{1},a_{2} \geq 0 \rbrace$$ is the set of all positive hyperbolic numbers. 
	We define $e_{1}$ and $e_{2}$ as $$ e_{1} =\frac{1+k}{2} , e_{2}= \frac{1-k}{2}$$ and satisfy following properties $$ e_{1}^{2}=e_{1},  e_{2}^{2}=e_{2}, e_{1}+e_{2}=1, e_{1}e_{2}=0, e_{1}^{*}=e_{1}, e_{2}^{*}=e_{2}$$
	In fact $e_{1}$ and $e_{2}$ are the idempotent basis of bi complex numbers and any bi complex number can be uniquely written as $$ w= z_{1} + z_{2}j =w_{1}e_{1}+w_{2}e_{2} $$ where, $$ w_{1}=z_{1}-z_{2}i,  w_{2}=z_{1}+z_{2}i.$$ 
	Also, $$ |w| = \frac{\sqrt{|w_{1}|^{2} + |w_{2}|^{2}}}{\sqrt{2}}.$$
	Let X be a $\mathbb{B}\mathbb{C}$ module then its idempotent representation is $$ X= e_{1}X_{1} \oplus e_{2}X_{2}=e_1{(e_1X)}\oplus e_2{(e_2X)} .$$
	\section{Bicomplex Inner Product}
	We now consider the bicomplex inner product.Let $X$ be $\mathbb{B} \mathbb{C}$ module. A map $$ \langle.,. \rangle : X \times X \rightarrow \mathbb{B} \mathbb{C}$$ is said to be a $\mathbb{B} \mathbb{C}$ inner product if it satisfies following properties for all $x,y,z \in X$ and $\alpha \in \mathbb{B} \mathbb{C}$:
	\begin{enumerate}
		\item $\langle x,y+z \rangle $ = $\langle x,y \rangle  + \langle x+z \rangle $
		\item $\langle \alpha x,y \rangle $ = $\alpha \langle x,y \rangle  $
		\item $\langle x,y \rangle $= $\langle y,x \rangle ^{*}$
		\item $\langle x,x \rangle  \in \mathbb{D^{+}}$ and $\langle x,x \rangle =0$ if and only if $x=0.$
	\end{enumerate}
	\begin{definition}
		A $\mathbb{B} \mathbb{C}$ module $X$ endowed with a bi complex inner product $ \langle .,. \rangle $ is said
		to be a $\mathbb B \mathbb C$ inner product module.
	\end{definition}
	Suppose $X_1$ and $X_2$ over $C[i]$ are equipped with the inner product $\langle .,. \rangle_1$ and $\langle .,. \rangle_2$ and corresponding norms 
	$\|.\|_1$ and $\|.\|_2.$ The formula $ \langle .,.\rangle_X= e_1 \langle x_1, w_1 \rangle_1 +e_2 \langle x_2, w_2 \rangle_2 $ defines bi complex inner product on the bi complex module $X.$
	The hyperbolic norm can be introduced on the inner product $\mathbb B \mathbb C$ module given by
	$$ \|x\|_{\mathbb D} = e_1 \|x_1\|_{1}+ e_2 \|x_2\|_2.$$
	It is also possible to introduce a real valued norm on Bicomplex module by setting
	$$ \|x\|_X^2= \frac{1}{2} ( \|x_1\|_1^2 + \|x_2\|_2^2).$$
	The relation between the two norms is $$|\|x\|_{\mathbb X}|=\|x\|_{\mathbb D}$$
	\begin{definition}
		A $ \mathbb B \mathbb C$ inner product module $X$ is said to be a bi complex Hilbert module if $X$ is complete with respect to the hyperbolic norm generated by the inner product square. 
	\end{definition} \noindent
	Any function $f:[a,b] \rightarrow \mathbb B \mathbb C$ can be written as 
	$$f(t) = x_0(t) + ix_1(t) +jx_2(t)+ij x_3(t)= $$
	$$[(x_0(t) +x_3(t)) + i(x_1(t) -x_2(t)) ]e_1 + [(x_0(t) -x_3(t)) + i(x_1(t) +x_2(t)) ]e_2=u(t) e_1 + v(t) e_2 $$ where the functions $u,v: [a,b] \rightarrow \mathbb{C} =C[i]$ .
	$$ |f(t)|^2=|x_0(t)|^2+|x_1(t)|^2+|x_2(t)|^2+|x_3(t)|^2.  $$
	Let $L^2[a,b]$ denote the space of bi complex valued functions that are square integrable 
	Let $ \mu$ denote the lebesgue measure on $ [a,b].$
	$$ \int_a^b |f(t)|^2 d \mu(t) < \infty.$$
	The inner product $$ \langle f, g \rangle = \int_a^b f(t) \overline{g(t)} d\mu(t) $$
	gives 
	$$ \langle u, u \rangle = \int_a^b |x_0(t) +x_3(t)|^2 +|x_1(t) -x_2(t)|^2 d\mu(t) $$
	and 
	$$  \langle v, v \rangle = \int_a^b |x_0(t) -x_3(t)|^2 +|x_1(t) +x_2(t)|^2 d\mu(t).$$
	$$\frac{ \|u\|_2^2+ \|v\|_2^2 }{2}= \int_a^b (|x_0(t)|^2+|x_1(t)|^2+|x_2(t)|^2+|x_3(t)|^2)d\mu(t) =\|f\|_X^2$$
	$$ \|f\|_X = \sqrt{\frac{\|u\|_2^2+ \|v\|_2^2 }{2}}$$
	$\mathbb{B} \mathbb{C}$ module $X$ can also be viewed as a normed vector space $X^\prime$ over $\mathbb{C}=C[i].$ (See \cite{Cha}.) $L^2[a,b]$ becomes a bicomplex Hilbert space with respect to the  norm of the associated vector space $L^2[a,b]^\prime$  over $ \mathbb{C}=C[i]$ is
 defined by $$ \int_a^b |f(t)|^2 d \mu(t).$$
	Note that $$ L^2 [a,b]^\prime = (e_1L^2[a,b]) +(e_2L^2[a,b])$$
	$L^2[a,b]$.
	Let $H$ be a bicomplex Hilbert space. Let $T:H \rightarrow $ be a bi complex linear operator over $H$. The adjoint of the operator is defined as 
		$$ \langle x, Ty \rangle = \langle T^* x,y \rangle $$ for all $x.y \in H$
	 Then the adjoint operator $T^*$ for $H$ considered as a bicomplex module is same as the adjoint of  operator on $H^\prime$ that is $H$ considered as $\mathbb{C}=C[i]$ vector space. $T$ is self adjoint if $T^*=T.$
	The bi complex space $L^2[a,b]$ is a $\mathbb{B} \mathbb{C}$ module. 
	\section{Main Results}
	\begin{definition}

		Let X  Bi complex Module. A map $T: X \rightarrow X$ is said to be linear if for all $a,b \in X$, we have $T(a+b)= T(a) + T(b)$ and $T(\mu a)=\mu T(a)$ for $\mu \in \mathbb{B}\mathbb{C}.$
	\end{definition}
	
	We now prove the special case of Spectral Theorem for the case when the bicomplex Hilbert space has a cyclic vector. A vector is said to be cyclic if the closure of the span of the elements $\lbrace x, Tx, T^2x, \ldots ,\rbrace $ is the whole of $X.$
	
	\noindent
	\begin{theorem}
		Let $H$ be a bi complex Hilbert module. Let $T$ be a linear operator on $H.$ Let  $H= e_{1}H_{1}+e_{2}H_{2}$ be its idempotent decomposition. Suppose $H$ has a cyclic vector then $H_{1}=e_1H$ and $H_{2}=eH$ also have a cyclic vector.
	\end{theorem}
	\begin{proof}
		Let $w$  be a cyclic vector in $H.$  Since $ H= e_{1}H_{1}+e_{2}H_{2}$,  $w=e_{1}w_{1}+e_{2}w_{2}=e_1(e_1w) +e_2(e_2w).$ $\overline{ Span \lbrace w, Tw,T^{2}w,\ldots \rbrace }=H.$\\ 
		Claim: $H_{1}$ has a cyclic vector $e_{1}w.$
		Let $z \in H_{1}$, so $z=e_{1}u$ , $u \in H.$ Since, $u \in \overline { Span \lbrace w, Tw.T^{2}w,...\rbrace}.$ So there exists a sequence $S_{n} \in span \lbrace w,Tw,T^{2}w,...\rbrace$ such that $S_{n} \rightarrow u=e_1(e_1u) +e_2(e_2u)=e_1u_1+e_2u_2.$ So,$$ S_{n}= \sum_{k=0}^{m(n)} \alpha_{k}T^{k}(w)$$ where $\alpha_{k}=e_{1}\alpha_{k}^{'}+e_{2}\alpha_{k}^{''}.$ Since
		$$T(w)=e_{1}T_{1}(e_1w)+e_{2}T_{2}(e_2w)$$ so, $$T^{k}(w)=e_{1}T_{1}^{k}(e_1w)+e_{2}T_{2}^{k}(e_2w)$$ for each $k \in N.$ Hence, $$ S_{n} = \sum_{k=0}^{m(n)} \alpha_{k}^{'}e_{1}T_{1}^{k}(e_1w) + \alpha_{k}^{''}e_{2}T_{2}^{k}(e_2w)= e_{1}\sum_{k=0}^{m(n)} \alpha_{k}^{'}T_{1}^{k}(e_1w) + e_2\sum_{k=0}^{m(n)} \alpha_{k}^{''}T_{2}^{k}(e_2w).$$ $$ = e_{1}S_{n}^{'}+e_{2}S_{n}^{''}.$$ Hence, $$| \| S_{n} - u\|_{\mathbb{D}}|^{2}=\| S_{n} - u \|^{2}= \frac{\sqrt{\| S_{n}^{'}-u_{1}\|_{1}^{2} + \| S_{n}^{''} -u_{2}\|_{2}^{2}}}{\sqrt{2}} \geq 0.$$ 
		Since $\| S_{n} - u\|_{\mathbb{D}}^{2} \rightarrow 0 $ we get $S_{n}^{'} \rightarrow u_{1}$ and $S_{n}^{''} \rightarrow u_{2}.$ Thus, $$ H_{1}= \overline{span \lbrace w_{1},Tw_{1},T^{2}w_{1},....\rbrace}$$ and $$H_{2}= \overline{span \lbrace w_{2},Tw_{2},T^{2}w_{2},...\rbrace}. $$ Therefore, $H_{1}$ has a cyclic vector $e_1w$ and $H_{2}$ has a cyclic vector $e_2w.$ 
	\end{proof}\noindent
	This following proposition is proved for the usual Hilbert space in Bhatia \cite{BR}.
	\begin{proposition}
         Suppose $T$ is a self adjoint operator on the usual Hilbert space $H$ with a cyclic vector $H.$ Then there exists a probability measure $ \mu$ on the interval $[-\|T\|,\|T\|]$ and a unitary operator $U$ from $H$ onto $L^2[-\|T\|,\|T\|]$ such that $UTU^*= M$ the canonical multiplication operator on $L^2[-\|T\|,\|T\|].$
	\end{proposition}
	\noindent
	We now have the following proposition.
	
	\begin{proposition}
		Suppose $T$ is a self adjoint operator on the bicomplex Hilbert module $H$ with a cyclic vector $H$. Then there exists a probability measures $ \mu_1$ and $\mu_2$  on the interval $[-\|T_1\|,\|T_1\|]$ and $[-\|T_2\|,\|T_2\|]$  unitary operators $U_1$ from $H_1$ onto $L^2[-\|T_1\|,\|T_1\|]$ and  $U_2$ from $H_2$ onto $L^2[-\|T_2\|,\|T_2\|]$ such that $UTU^*= M$, where $U=e_1U_1+e_2 U_2$, $ T= e_1T_1+e_2T_2 $ and $ M= e_1 M_1 +e_2M_2.$
	\end{proposition}
	\begin{proof}
		Since $H_1$ and $H_2$ are complex Hilbert spaces, $T_1:H_1 \rightarrow H_1$ and $T_2:H_2 \rightarrow H_2$ are self adjoint operators. By the above proposition, There exists probability measures $\mu_1, \mu_2$ on $[-\|T_1\|,\|T_1\|] $ and $[-\|T_2\|,\|T_2\|] $ and unitary operators $U_1: H_1 \rightarrow L^2( \mu_1)$ and $U_2:H_2 \rightarrow L^2( \mu_2)$ such that $U_1T_1U_1^*=M_1$ and  $U_2T_2U_2^*=M_2$. $M_1$ and $M_2$ are canonical multiplication operators in $L^2( \mu_1)$ and $L^2( \mu_2)$  respectively. Set $U=e_1U_1 +e_2U_2$ and $T=e_1T_1 +e_2T_2$ the idempotent decomposition. We then have $$ U^*TU=M=e_1M_1+e_2M_2.$$
	\end{proof}
	We now prove the spectral theorem for a bounded self adjoint operator.
	\begin{theorem}
		Let $T$ be a bounded self adjoint operator on a separable bicomplex, countabely generated Hilbert space $H.$ Then there exists probabilty measures $\mu_1,\mu_2,\ldots, \mu_1^\prime,\mu_2^\prime, \ldots  $ on the intervals $[-\|T_1\|,\|T_1\|] $ and $[-\|T_2\|,\|T_2\|] $ and unitary operators $U_1(n): H_1 \rightarrow L^2( \mu_n)$ and  $U_2(n): H_2 \rightarrow L^2( \mu_n^\prime )$ such that
		$$UTU^*= M$$ where $M$ is the multiplication operator on $[e_1L^2(\mu_1) + e_2 L^2(\mu_1^\prime)]\oplus[e_1L^2(\mu_2) \oplus e_2 L^2(\mu_2^\prime)] \oplus \ldots $
	\end{theorem}
	\begin{proof}
		Let $\|x_1\|_X=1$ Let $$S_1=\overline{ Span \lbrace x_1, Tx_1,T^{2}x_1,\ldots \rbrace }$$ If $S_1=H$ we are done by the above proposition. Else note that
		$S_1$ is $T$ invariant. Consider $y=T(u)$, $u \in S_1$ There is a sequence $T^nx_1 \rightarrow u$. But $T$ is bounded implies $T^{n+1} x_1 \rightarrow Tu$. This means that $y=Tu \in \overline{S_1}.$ Also note that
		$ S_1^\perp= \lbrace x \in H : \langle x, y \rangle = 0, \text{for all } y \in S^1 \rbrace $ is invariant under $T$. Let $z=T(v) \in T(S_1^\perp) $. For $x \in S^1,$ $x= \lim T^nx_1$
		
		$$ \langle x,z \rangle =\langle \lim T^n x_1, Tv \rangle =\lim \langle T^* T^n x_1, v \rangle =0 $$
		Let $x_2 \in S_1^\perp $ and $x_2$ be a non null vector $\langle x_2e_1,x_2e_1 \rangle_1 \neq 0$(See \cite{Ger}) and  $\langle x_2e_2,x_2e_2 \rangle_2 \neq 0.$ So that $\|x_2\|_X=1.$ and let $$S_2=\overline{ Span \lbrace x_2, Tx_2,T^{2}x_2,\ldots \rbrace }.$$ An application of Zorn's Lemma show that $H$ can be written as a countable direct sum of $S_1.S_2,\ldots$ in which each $S_n$ is a cyclic sub module and $S_n$ are mutually orthogonal.The above proposition can be used to get the measures $\mu_n, \mu_n^\prime $ to get the result.
		
	\end{proof}
	\noindent
	Acknowledement : The author Akshay S. Rane would like to thank UGC faculty recharge program, India for their support.

\end{document}